\documentclass[12pt]{article}

\usepackage{amssymb}
\usepackage{latexsym, amsmath, amscd,amsthm}
\usepackage{rotating}
\usepackage[all]{xy}
\usepackage{graphicx,epsf}
\usepackage{psfrag}
\usepackage{amssymb}
\usepackage{multirow}
\usepackage{rotating}
\usepackage{amsthm}
\usepackage[all]{xy}
\usepackage{amsmath}
\usepackage{amssymb}
\usepackage{multirow}
\usepackage{amsthm}

\newtheorem{theorem}{Theorem}[section]

\newtheorem{proposition}[theorem]{Proposition}
\newtheorem{lemma}[theorem]{Lemma}

\theoremstyle{definition}
\newtheorem{definition}[theorem]{Definition}
\newtheorem{example}[theorem]{Example}
\newtheorem{remark}[theorem]{Remark}

\def\ord{\mathop{\operatorfont ord}\nolimits}

\def\diag{\mathop{\rm diag}\nolimits}

\newcommand{\bc}{{\mathbf{c}}}

\begin{document}

\title{Minimum distance of Hermitian two-point codes}

\author {Seungkook Park\footnotemark[1]}

\date{}

\renewcommand{\thefootnote}{\fnsymbol{footnote}}
\footnotetext[1]{Department of Mathematical Sciences, University of
Cincinnati (seung-kook.park@uc.edu)}
\renewcommand{\thefootnote}{\arabic{footnote}}

\maketitle

\begin{abstract}
We prove a formula for the minimum distance of two-point codes on a
Hermitian curve.
\end{abstract}

\section{Introduction}
Homma and Kim \cite{HommaKim4},\cite{HommaKim1},\cite{HommaKim3},\cite{HommaKim2}
gave a complete determination of the minimum distance of two-point
codes on a Hermitian curve. Their method is based on the method of
Kumar and Yang\cite{KumarYang}. It leads to a theorem with many distinct cases and a
long proof divided over four papers. The objective of this paper is to give a short and
easy proof of the minimum distance of the Hemitian two-points codes using a method based on
Kirfel and Pellikaan \cite{kirfelPellikaan}. First we use the shift
bound to find the lower bound for the minimum distance. Secondly we use certain types of conics and
lines to show that the bound is sharp. Kirfel and Pellikaan gave a general method of finding a lower bound for the minimum distance of one-point codes based on the decoding algorithms by Feng and Rao \cite{FengRao}, and Duursma \cite{Majoritycoset}. The method gives a short proof for the minimum distance of the Hermitian one-point codes. There have been approaches using Kirfel and Pellikaan's method to find the
lower bound for the minimum distance of general algebraic geometric codes. In \cite{norder}, near order functions are used to find the lower bound for the minimum distance of
a two-point algebraic geometric code. Beelen \cite{Beelen} gave a
method of finding the lower bound for the minimum distance of
general algebraic geometric codes using the generalized order bound.
In fact, the lower bound gives the exact minimum distance of the
Hermitian two-point codes in a large range. We define the multiplicity of the Hermitian
two-point code and use the multiplicity to find a path that will give the minimum distance
of the Hermitian two-point code. Our formulas for the minimum distance of the Hermitian two-point code are given for all ranges and they meet the formulas by Homma and Kim with a shorter proof and fewer cases for the formulas. Moreover, our approach can be used in majority coset
decoding \cite{Majoritycoset} which decodes up to half the actual minimum distance. \\

 In Section 2 we give the definition of multiplicity and a method
  of finding the lower bounds for the minimum distance
of Hermitian two-point codes. In Section 3 we state the formulas for
the multiplicity and minimum distance of the Hermitian two-point
codes. The proofs of the formulas for the multiplicity and minimum
distance of the Hermitian two-point codes are given in Section 4 and
Section 5, respectively. In Appendix, we state both
formulas for the minimum distance of Hermitian two-point codes given
by Homma and Kim
\cite{HommaKim4},\cite{HommaKim1},\cite{HommaKim3},\cite{HommaKim2}
and given in this paper for comparison.

\section{Multiplicities and the shift bound}\label{multishift}

We give the definition of multiplicity and a method of finding the
lower bound for the minimum distance of the Hermitian two-point code.\\

Let $X$ be a Hermitian curve defined by $y^q+y=x^{q+1}$ over
$\mathbb{F}_{q^2}$. Then $X$ has $q^3+1$ rational points and the
genus is $q(q-1)/2$. Let $P_\infty$ be the point at infinity of $X$
and $P_0$ the origin of $X$. The canonical divisor $K$ of a
Hermitian curve is $K=(q-2)H$, where $H \sim (q+1)P_\infty \sim
(q+1)P_0$. Let $\mathbb{F}_{q^2}(X)$ be the function field of $X$
over $\mathbb{F}_{q^2}$. For $f\in \mathbb{F}_{q^2}(X)\backslash
\{0\}$, $(f)_\infty$ denotes the pole divisor of $f$, $(f)_0$ the
zero divisor of $f$ and $(f)=(f)_0-(f)_\infty$ the divisor of $f$.
Given a divisor $G$ on $X$ defined over $\mathbb{F}_{q^2}$, let
$L(G)$ denote the vector space over $\mathbb{F}_{q^2}$ consisting of
functions $f\in \mathbb{F}_{q^2}(X)\backslash\{0\}$ with $(f)+G\geq
0$ and the zero function.
Let $G=aP_\infty+bP_0$ and $D=P_1+\cdots +P_n$ be a divisor of $X$,
where $\operatorname{supp}(G)\cap \operatorname{supp}(D)=\emptyset$
and $P_1,\ldots, P_n$ are pairwise distinct. We define a code
$C(D,G)$ as the image of the evaluation map $\rm{ev}:L(G)\rightarrow
\mathbb{F}_{q^2}^n$ given by $\rm{ev}(f)=(f(P_1),\ldots, f(P_n))$
for all $f\in L(G)$.
For a fixed $D$, we will use the notation $C(G)=C(a,b)$ for
$C(D,G)$, where $G=aP_\infty+bP_0$. \\

\begin{definition}
Let $M_{P_\infty}(a,b)$ be the set of pairs $(f,g)$ of rational
functions such that
\begin{itemize}
\item[(1)] $fg \in L(aP_\infty+bP_0)\backslash L((a-1)P_\infty+bP_0)$
\item[(2)] $f \in L(aP_\infty+bP_0)$
\item[(3)] $g \in L((a+b)P_\infty)$
\end{itemize}
The multiplicity $m_{P_\infty}(a,b)$ is defined as
\[
m_{P_\infty}(a,b) = \# \{ -b \leq i \leq a+1 : \exists (f,g) \in
M_{P_\infty}(a+1,b) \text{ with } \ord_{P_\infty}(f) = -i \}.
\]
\end{definition}
\vspace{0.8cm}

 We apply the shift bound argument from \cite
{vanlintwilson}, see also \cite{Pellikaan}, to obtain a lower bound for the weight of a vector
that is orthogonal
to $C(a,b)$ but not orthogonal to $C(a+1,b)$ in terms of the multiplicity. \\

\begin{theorem}\label{shift}
Let $\bc = (c_1,\ldots,c_n) \in \mathbb{F}_{q^2}^n$ be a vector that
is orthogonal to $C(a,b)$ but not orthogonal to $C(a+1,b)$. Then the
weight of $\bc$ is at least $m_{P_\infty}(a,b)$.
\end{theorem}

\begin{proof}
Let $m = m_{P_\infty}(a,b)$ and let $ (f_1,g_1), \ldots, (f_m,g_m)$
be pairs in $M_{P_\infty}(a+1,b)$ such that $f_1, f_2, \ldots, f_m$
have distinct pole orders at $P_\infty.$ Let
 $\ord_{P_\infty}(f_m) < \cdots < \ord_{P_\infty}(f_1)$. Then
\begin{align*}
& g_jf_i \in L(G)~~\text{ for}~~ i=1,\ldots ,j-1.\\
& g_jf_i \in L(G+P_{\infty})\backslash L(G)~~\text{ for}~~ i=j.\\
& g_jf_i \not \in L(G+P_{\infty})~~\text{ for}~~ i=j+1,\ldots ,m.
\end{align*}
Let $A$ be the $m \times n$ matrix with entries $g_i(P_j)$ and let
$B$ be the $m \times n$ matrix with entries $f_i(P_j)$. The $m
\times m$ matrix $A \diag(c_1,\ldots,c_n) B^T$ is zero below the
diagonal and nonzero on the diagonal. Hence it is of rank $m$ and
the number of nonzero coordinates in $\bc$ is at least $m$.
\end{proof}

\begin{definition}
Let $M_{P_0}(a,b)$ be the set of pairs $(f,g)$ of rational functions
such that
\begin{itemize}
\item[(1)] $fg \in L(aP_\infty+bP_0)\backslash L(aP_\infty+(b-1)P_0)$
\item[(2)] $f \in L(aP_\infty+bP_0)$
\item[(3)] $g \in L((a+b)P_0)$
\end{itemize}
The multiplicity $m_{P_0}(a,b)$ is defined as
\[
m_{P_0}(a,b) = \# \{ -a \leq j \leq b+1 : \exists (f,g) \in
M_{P_0}(a,b+1) \text{ with } \ord_{P_0}(f) = -j \}.
\]
\end{definition}

\begin{theorem}
Let $\bc = (c_1,\ldots,c_n) \in \mathbb{F}_{q^2}^n$ be a vector that
is orthogonal to $C(a,b)$ but not orthogonal to $C(a,b+1)$. Then the
weight of $\bc$ is at least $m_{P_0}(a,b)$.
\end{theorem}

\begin{proof}
The proof is similar to Theorem \ref{shift}.
\end{proof}

We present a method to find a lower bound for the minimum distance
of the Hermitian two-point codes. Let $0 \neq c \in C(a,b)^\perp$.
We want to find a lower bound for the weight of a word $c \neq 0$
which is orthogonal to $C(a,b)$. Consider Figure \ref{path}.\\

\begin{figure}[!ht]
\centering
 \psfrag{c1}{$C(a,b)$} \psfrag{l1}{$m_{P_\infty}(a,b)$}
 \psfrag{r1}{$m_{P_0}(a,b)$}
\psfrag{cl1}{$C(a+1,b)$}
 \psfrag{cr1}{$C(a,b+1)$}
\psfrag{cl2}{$C(a+2,b)$}
 \psfrag{c2}{$C(a+1,b+1)$}
\psfrag{cr2}{$C(a,b+2)$}
 \psfrag{f}{$\mathbb{F}_q^n$}
 \psfrag{f1}{$C(a+l_1,b)=\mathbb{F}_q^n$}
 \psfrag{f2}{$\mathbb{F}_q^n=C(a,b+l_2)$}
\includegraphics[scale=1]{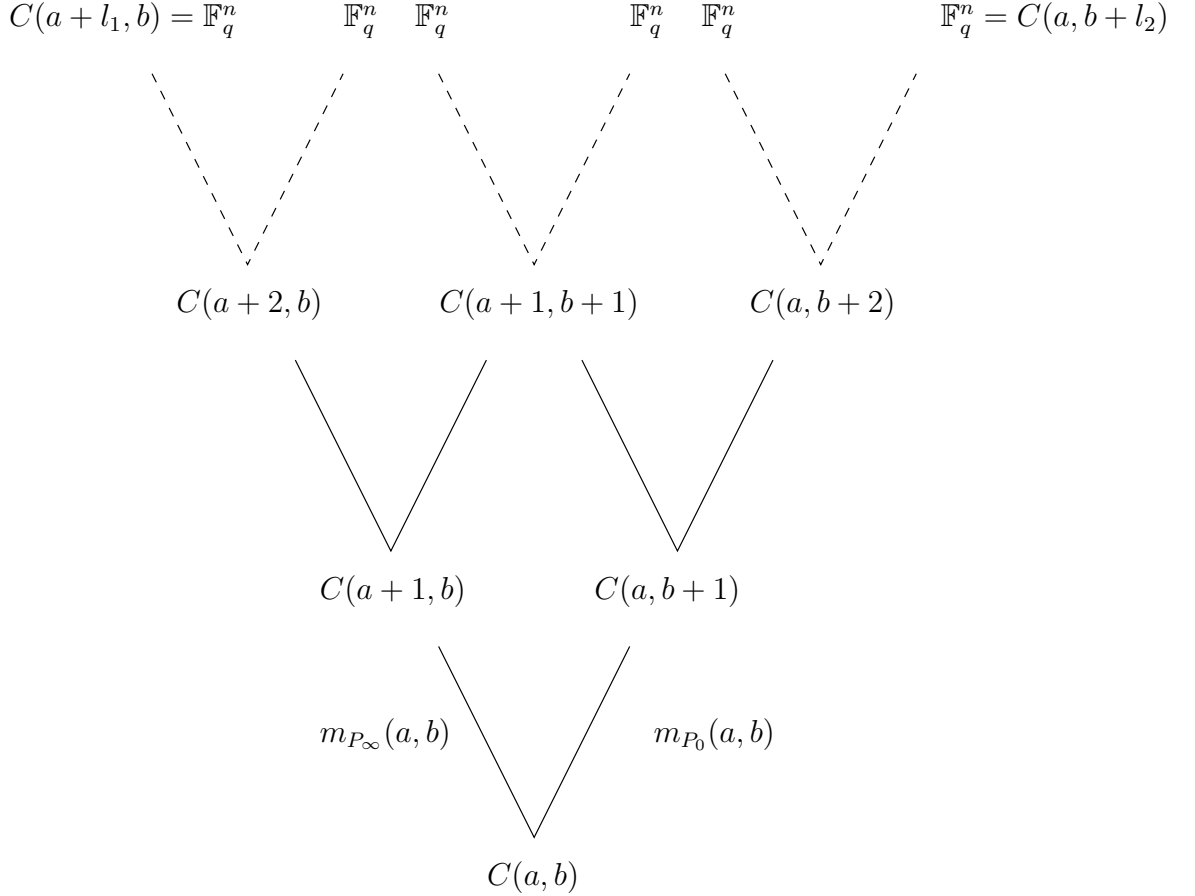}
\caption{Lower bound for the minimum distance}\label{path}
\end{figure}

We first consider the vector spaces
\[C(a,b) \subseteq C(a+1,b) \subseteq \cdots
\subseteq \mathbb{F}_q^n.
\]
The weight for the words $c$ that are orthogonal to $C(a,b)$ but not
orthogonal to $C(a+1,b)$ is at least $m_{P_\infty}(a,b)$. The weight
for the words $c$ that are orthogonal to $C(a+1,b)$ but not
orthogonal to $C(a+2,b)$ is at least $m_{P_\infty}(a+1,b)$. Since
$c$ is nonzero, there is a vector space not equal to the full space
$\mathbb{F}_q^n$ , say $C(a+i,b)$, such that the word $c$ is
orthogonal to $C(a+i,b)$ but not orthogonal to $C(a+i+1,b)$. The
weight of the word $c$ that is orthogonal to $C(a,b)$ but not
orthogonal to $\mathbb{F}_q^n$ is at least the minimum of the
multiplicities $m_{P_\infty}(a+j,b)$, where $j = 0,1,\ldots, l_1$
and $C(a+l_1,b)=\mathbb{F}_q^n$. Now, we fix $a$ and increase $b$.
The weight for the words $c$ that are orthogonal to $C(a,b)$ but not
orthogonal to $C(a,b+1)$ is at least $m_{P_0}(a,b)$. The weight for
the words $c$ that are orthogonal to $C(a,b+1)$ but not orthogonal
to $C(a,b+2)$ is at least $m_{P_0}(a,b+1)$. By a similar argument as
above, the weight of the word $c$ that is orthogonal to $C(a,b)$ but
not orthogonal to $\mathbb{F}_q^n$ is at least the minimum of the
multiplicities $m_{P_0}(a,b+j)$, where $j = 0,1,\ldots, l_2$ and
$C(a,b+l_2)=\mathbb{F}_q^n$. We increase the divisors by increasing
the pole order of $P_0$ or $P_\infty$ by 1, and compute the
$m_{P_0}(a+1,b),~m_{P_\infty}(a+1,b), m_{P_0}(a,b+1),\text{ and}
~~m_{P_\infty}(a,b+1)$. We apply the same process until the
Riemann-Roch space of the divisor becomes the full space
$\mathbb{F}_q^n$. In each step, we can make a choice for the divisor
by adding $P_0$ or $P_\infty$, that is, we can choose a path to the
full space $\mathbb{F}_q^n$. For each path $P$, we take the minimum
of the multiplicities along the path and denote it by $\min(P)$. Let
$S$ be the set of $\min(P)$ for all the paths. Each element of the
set $S$ gives a lower bound for the weight of $c$ with $0 \neq c
\perp C(a,b)$. The best lower bound for the weight of $c$ is
obtained by taking the maximum of the set $S$. This maximum is a
lower bound for the minimum distance of $C(a,b)^\perp$.\\
\section{Formulas for multiplicity and minimum distance}\label{Hermitian}
We state the formulas for the multiplicity and minimum distance of
the Hermitian two-point codes. The formulas give the minimum distance of the
Hermitian two-point codes for all ranges of $G$. We divide the ranges into two parts as follows:
\begin{itemize}
\item[$1$.]$\{G:\rm{deg} G > \rm{deg} K + q\} ~\cup ~\{G:\rm{deg} K \leq \rm{deg} G \leq \rm{deg} K+q ~\wedge G \nsim sP_\infty ~\wedge G \nsim tP_0\}$ and \\
\item[$2$.]$\{G:\rm{deg} G < \rm{deg} K\} ~\cup~ \{G:(\rm{deg} K \leq \rm{deg} G \leq \rm{deg} K+q) ~\wedge~ (G \sim sP_\infty ~\vee~ G \sim tP_0)\}$, where $s,t \in \mathbb{Z}$.
\end{itemize}
Theorem \ref{thm:dist} and Theorem \ref{thm:below} give the formulas
for the minimum distance for the first part and the second part, respectively.
\begin{proposition}\label{prop:mult}
Let $G = K+aP_\infty+bP_0$, where $K$ is a canonical divisor,
\begin{eqnarray*}
a&=& a_0 (q+1)-a_1,~~~~0\leq a_1\leq q,\\
b&=& b_0 (q+1)-b_1,~~~~0\leq b_1\leq q.
\end{eqnarray*}
\begin{itemize}
\item[$1$.] If $a_1< a_0+b_0$, then \\
$m_{P_\infty}(2g-2+a,b)$ = $(a_0+b_0-a_1)(q+1)-b_1+a_1q$
\item[$2$.] If $a_0+b_0\leq a_1 \leq a_0+b_0+q-1$, then \\
$m_{P_\infty}(2g-2+a,b)$ = $a_1(q+a_0+b_0-a_1)-\min\{a_1,b_1\}$.
\item[$3$.] If $a_0+b_0+q-1<a_1$, then \\
$m_{P_\infty}(2g-2+a,b)$ = $0$.
\item[$1^{'}$.] If $b_1< a_0+b_0$, then \\
 $ m_{P_0}(2g-2+a,b)$ = $(a_0+b_0-b_1)(q+1)-a_1+b_1q$
\item[$2^{'}$.] If $a_0+b_0\leq b_1 \leq a_0+b_0+q-1$, then \\
 $ m_{P_0}(2g-2+a,b)$ = $b_1(q+a_0+b_0-b_1)-\min\{a_1,b_1\}$.
\item[$3^{'}$.] If $a_0+b_0+q-1<b_1$, then \\
$m_{P_0}(2g-2+a,b)$ = $0$.
\end{itemize}
\end{proposition}
\begin{proof}
The proof is given in Section \ref{sec:pfmult}
\end{proof}
\begin{theorem}\label{thm:mult}
Let $G = K+aP_\infty+bP_0$, where $K$ is a canonical divisor,
\begin{eqnarray*}
a&=& a_0 (q+1)-a_1,~~~~0\leq a_1\leq q,\\
b&=& b_0 (q+1)-b_1,~~~~0\leq b_1\leq q.
\end{eqnarray*}
Let $d^*=\deg(G) -(2g-2)=a+b$.
\begin{itemize}
\item[$1$.] If $a_1< a_0+b_0$, then \\
$m_{P_\infty}(2g-2+a,b)$ = $d^*$
\item[$2$.] If $a_0+b_0\leq a_1 \leq a_0+b_0+q-1$, then \\
$m_{P_\infty}(2g-2+a,b)$ =
$d^*+(a_1-a_0-b_0)(q+1-a_1)+\max\{0,b_1-a_1\}$.
\item[$3$.] If $a_0+b_0+q-1<a_1$, then \\
$m_{P_\infty}(2g-2+a,b)$ = $0$.
\item[$1^{'}$.] If $b_1< a_0+b_0$, then \\
 $ m_{P_0}(2g-2+a,b)$ = $d^*$
\item[$2^{'}$.] If $a_0+b_0\leq b_1 \leq a_0+b_0+q-1$, then \\
 $ m_{P_0}(2g-2+a,b)$ = $d^*+(b_1-a_0-b)(q+1-b_1)+\max\{0,a_1-b_1\}$.
\item[$3^{'}$.] If $a_0+b_0+q-1<b_1$, then \\
$m_{P_0}(2g-2+a,b)$ = $0$.
\end{itemize}
\end{theorem}
\begin{proof}
Follows from Proposition \ref{prop:mult}.
\end{proof}
\begin{theorem}\label{thm:dist}
Suppose that $G$ satisfies either
\begin{itemize}
\item[$(1)$]$\rm{deg} G > \rm{deg} K + q$ or
\item[$(2)$]$\rm{deg} K \leq \rm{deg} G \leq \rm{deg} K+q ~\text{and}~ G \nsim sP_\infty ~\text{and} ~G \nsim tP_0$ for all $s,t \in \mathbb{Z}$.
\end{itemize}
Let $G = K+aP_\infty+bP_0$, where $K$ is a
canonical divisor,
\begin{eqnarray*}
a&=& a_0 (q+1)-a_1,~~~~0\leq a_1\leq q,\\
b&=& b_0 (q+1)-b_1,~~~~0\leq b_1\leq q.
\end{eqnarray*}
Let $d^*=\deg(G) -(2g-2)=a+b$.
\begin{itemize}
\item[$1$.] If $0\leq a_1,b_1\leq a_0+b_0$, then\\
$d(C(D,G)^\perp)=d^*$.
\item[$2$.] If $0\leq b_1\leq a_0+b_0<a_1$, then\\
$d(C(D,G)^\perp)=d^*+a_1-(a_0+b_0)$.
\item[$2^{'}$.] If $0\leq a_1\leq a_0+b_0<b_1$, then\\
$d(C(D,G)^\perp)=d^*+b_1-(a_0+b_0)$.
\item[$3$.] If $a_0+b_0<a_1 \leq b_1<q$, then\\
$d(C(D,G)^\perp)=d^*+a_1+b_1-2(a_0+b_0)$.
\item[$3^{'}$.] If $a_0+b_0<b_1\leq a_1<q$, then\\
$d(C(D,G)^\perp)=d^*+a_1+b_1-2(a_0+b_0)$.
\item[$4$.] If $a_0+b_0<a_1,b_1$ and $a_1=q$, $b_1=q$, then\\
$d(C(D,G)^\perp)=d^*+q-(a_0+b_0)$.
\end{itemize}
\end{theorem}
\begin{proof}
The proof is given in Section \ref{sec:pfdist}
\end{proof}
\begin{remark}
We can rewrite Theorem \ref{thm:dist} as
$$
d(C(D,G)^\perp)=d^*+\max\{0,a_1-(a_0+b_0),
b_1-(a_0+b_0),a_1+b_1-2(a_0+b_0)\},
$$
for all cases except case $4$.
\end{remark}
\begin{theorem}\label{thm:below}
Suppose that $G$ satisfies either
\begin{itemize}
\item[$(1)$]$\rm{deg} G < \rm{deg} K$ or
\item[$(2)$]$\rm{deg} K \leq \rm{deg} G \leq \rm{deg} K+q ~\text{with}~ G \sim sP_\infty ~\text{or}~ G \sim tP_0$ for some $s,t \in \mathbb{Z}$.
\end{itemize}
If $G = aP_\infty+bP_0$ with
\begin{eqnarray*}
a&=& a_0 (q+1)+a_1,~~~~0\leq a_1\leq q,\\
b&=& b_0 (q+1)+b_1,~~~~0\leq b_1\leq q.
\end{eqnarray*}
then $$d(C(D,G)^\perp)=a_0+b_0+2.$$
\end{theorem}
\begin{proof}
The proof is given in Section \ref{sec:pfdist}
\end{proof}



\section{Proof of Proposition \ref {prop:mult}}\label {sec:pfmult}
\ \\ \noindent{\bf Proposition 3.1} {\it Let $G = K+aP_\infty+bP_0$,
where $K$ is a canonical divisor,
\begin{eqnarray*}
a&=& a_0 (q+1)-a_1,~~~~0\leq a_1\leq q,\\
b&=& b_0 (q+1)-b_1,~~~~0\leq b_1\leq q.
\end{eqnarray*}
\begin{itemize}
\item[$1$.] If $a_1< a_0+b_0$, then \\
$m_{P_\infty}(2g-2+a,b)$ = $(a_0+b_0-a_1)(q+1)-b_1+a_1q$
\item[$2$.] If $a_0+b_0\leq a_1 \leq a_0+b_0+q-1$, then \\
$m_{P_\infty}(2g-2+a,b)$ = $a_1(q+a_0+b_0-a_1)-\min\{a_1,b_1\}$.
\item[$3$.] If $a_0+b_0+q-1<a_1$, then \\
$m_{P_\infty}(2g-2+a,b)$ = $0$.
\item[$1^{'}$.] If $b_1< a_0+b_0$, then \\
 $ m_{P_0}(2g-2+a,b)$ = $(a_0+b_0-b_1)(q+1)-a_1+b_1q$
\item[$2^{'}$.] If $a_0+b_0\leq b_1 \leq a_0+b_0+q-1$, then \\
 $ m_{P_0}(2g-2+a,b)$ = $b_1(q+a_0+b_0-b_1)-\min\{a_1,b_1\}$.
\item[$3^{'}$.] If $a_0+b_0+q-1<b_1$, then \\
$m_{P_0}(2g-2+a,b)$ = $0$.
\end{itemize}}

\begin{figure}[!ht]
\centering \psfrag{q}{$q$} \psfrag{b1}{$b_1$} \psfrag{a1}{$a_1$}
\psfrag{a0+b0-a1}{$a_0+b_0-a_1$}
\psfrag{a0+b0+q-1}{$a_0+b_0+q-1$}\psfrag{q+1}{$q+1$}
\includegraphics[scale=0.5]{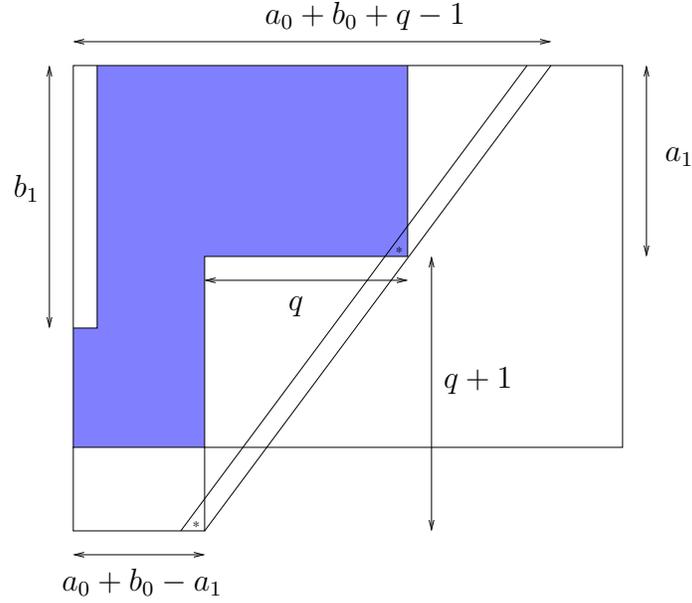}
\caption{Counting multiplicity}\label{mult}
\end{figure}
\begin{proof}
By renaming $P_0$ with $P_\infty$ and $P_\infty$ with $P_0$, it is
enough to prove the cases 1,2 and 3. We may assume that
$K=(2g-2)P_\infty$. Let $G=K+aP_\infty+bP_0$, where
\begin{eqnarray*}
a&=& a_0 (q+1)-a_1,~~~~0\leq a_1\leq q,\\
b&=& b_0 (q+1)-b_1,~~~~0\leq b_1\leq q.
\end{eqnarray*}
Then $L(K+aP_\infty+bP_0) = L((q-2+a_0+b_0)(q+1)P_\infty-a_1
P_\infty - b_1 P_0)$ is spanned by the monomials $x^i y^j$ with
\begin{align*}
   (1)~~ &0 \leq i \leq q,~~0 \leq j,~~ i+j \leq q-2+a_0+b_0, \\
   (2)~~ &i \geq a_1 ~~\text{ for }~~ i+j = q-2+a_0+b_0,\\
   (3)~~ &i \geq b_1 ~~\text{ for }~~ j=0. \\
\end{align*}
We determine the pairs of integers $(i,j)$ such that there exist
\begin{align*}
&f \in L(iP_\infty+bP_0) \backslash L((i-1)P_\infty+bP_0)\\
&g \in L(jP_\infty) \backslash L((j-1)P_\infty),
\end{align*}
with $fg \in L(K+(a+1)P_\infty+bP_0)\backslash L(K+aP_\infty+bP_0)$.
Let $x^{i_1} y^{j_1}$ and $x^{i_2} y^{j_2}$ be the leading monomials
of $f$ and $g$, respectively, with $~0 \leq i_1, ~i_2 \leq q,~~ 0
\leq j_1,~ j_2$. The product $fg \in L(K+(a+1)P_\infty+bP_0)
\backslash L(K+aP_\infty+bP_0)$ if
\begin{align*}
&(1)~~ x^{i_1} y^{j_1} x^{i_2} y^{j_2} = x^{a_1-1} y^{q-1+a_0+b_0-a_1},~~ \text{or} \\
&(2)~~ x^{i_1} y^{j_1} x^{i_2} y^{j_2} = x^{q+a_1}
y^{-1+a_0+b_0-a_1},
\end{align*}
The solutions for $(i_1,j_1)$ are
\begin{align*}
&(1)~~ 0 \leq i_1 \leq a_1-1,~~~ 0 \leq j_1 \leq q-1+a_0+b_0-a_1,
~~\text{such that}
     ~~i_1 \geq b_1 ~~\text{for}~~ j_1 = 0. \\
&(2)~~ a_1 \leq i_1 \leq q,~~~0 \leq j_1 \leq -1+a_0+b_0-a_1.
\end{align*}
or
\begin{align*}
&(1') ~~0 \leq j_1 \leq -1+a_0+b_0-a_1,~~~ 0 \leq i_1 \leq q,
~~\text{such
 that}~~ i_1 \geq b_1 ~~\text{for} ~~j_1 = 0. \\
&(2')~~ a_0+b_0-a_1 \leq j_1 \leq q-1+a_0+b_0-a_1,~~~0 \leq i_1 \leq
a_1 - 1.
\end{align*}
If  $a_0+b_0+q-1\geq a_1 \geq a_0+b_0$, then the total number of
pairs $(i,j)$ is $(a_0+b_0-a_1)(q+1)-b_1+q a_1. $ If $a_1 >
a_0+b_0+q-1$, then (1) and (2) have no solutions. For
$a_0+b_0-a_1\leq 0$, if $a_1=0$, then (1) and (2) have no solutions.
If $a_1\neq 0$, then there are no solutions in (2). In (1), there
are
$$
a_1(q+a_0+b_0-a_1)-\min\{a_1,b_1\}
$$
solutions. Thus we have the multiplicity
$a_1(q+a_0+b_0-a_1)-\min\{a_1,b_1\}$ if $a_0+b_0-a_1\leq 0$.
\end{proof}


\section{Proof of Theorem \ref{thm:dist}, \ref{thm:below}}\label{sec:pfdist}
For each path, the minimum of the multiplicities along the path is a
lower bound for the minimum distance of $C(D,G)^\perp$. In Theorem
\ref{thm:dist}, we find a path that gives a lower bound of the
minimum distance which is sharp. The following two lemmas give the
minimum of the multiplicities of a certain part of the path chosen
in Theorem \ref{thm:dist}.

\begin{lemma}\label{mult2}
Suppose that $G$ satisfies either
\begin{itemize}
\item[$(1)$]$\rm{deg} G > \rm{deg} K + q$ or
\item[$(2)$]$\rm{deg} K \leq \rm{deg} G \leq \rm{deg} K+q ~\text{and}~ G \nsim sP_\infty ~\text{and} ~G \nsim tP_0$ for all $s,t \in \mathbb{Z}$.
\end{itemize}
Let $G = K+aP_\infty+bP_0$, where $K$ is a
canonical divisor,
\begin{eqnarray*}
a&=& a_0 (q+1)-a_1,~~~~0\leq a_1\leq q,\\
b&=& b_0 (q+1)-b_1,~~~~0\leq b_1\leq q.
\end{eqnarray*}
Let $d^*=\deg(G) -(2g-2)=a+b$.\\
If $0\leq b_1\leq a_0+b_0\leq a_1$ and
$I_2=\{(2g-2+a,b),(2g-2+a+1,b),\ldots,(2g-2+a+a_1-(a_0+b_0),b)\}$,
then
$$\min_{i \in I_2}~(m_{P_\infty}(i))=d^*+a_1-(a_0+b_0)=a_0q+b_0q-a_1.$$\\
If $0\leq a_1\leq a_0+b_0\leq b_1$ and
$I_{2^{'}}=\{(2g-2+a,b),(2g-2+a,b+1),\ldots,(2g-2+a,b+b_1-(a_0+b_0))\}$,
then
$$\min_{i \in I_{2^{'}}}~(m_{P_0}(i))=d^*+b_1-(a_0+b_0)=a_0q+b_0q-b_1.$$
\end{lemma}
\begin{proof}
For the case  $0\leq b_1\leq a_0+b_0\leq a_1$, we need to show that
$m_{P_\infty}(i)\geq d^*+a_1-(a_0+b_0)$ for $i \in I_2$, that is, we
need to show that $m_{P_\infty}(2g-2+a_0(q+1)-a_1^{'},b) \geq
d^*+a_1-(a_0+b_0)$ for $a_1^{'}=a_1, a_1-1, \ldots, a_0+b_0$. By
Proposition \ref{prop:mult}, we have
\begin{align*}
&m_{P_\infty}(2g-2+a_0(q+1)-a_1^{'},b)-(d^*+a_1-(a_0+b_0))\\
=&(a_1^{'}-(a_0+b_0))(q-a_1^{'})\geq 0.
\end{align*}
The other case follows by symmetry.
\end{proof}

\begin{lemma}\label{mult3}
Suppose that $G$ satisfies either
\begin{itemize}
\item[$(1)$]$\rm{deg} G > \rm{deg} K + q$ or
\item[$(2)$]$\rm{deg} K \leq \rm{deg} G \leq \rm{deg} K+q ~\text{and}~ G \nsim sP_\infty ~\text{and} ~G \nsim tP_0$ for all $s,t \in \mathbb{Z}$.
\end{itemize}
Let $G = K+aP_\infty+bP_0$, where $K$ is a
canonical divisor,
\begin{eqnarray*}
a&=& a_0 (q+1)-a_1,~~~~0\leq a_1\leq q,\\
b&=& b_0 (q+1)-b_1,~~~~0\leq b_1\leq q.
\end{eqnarray*}
Let $d^*=\deg(G) -(2g-2)=a+b$.\\
If $a_0+b_0 \leq a_1 \leq b_1<q$ and
$I_3=\{(2g-2+a,b),(2g-2+a+1,b),\ldots,(2g-2+a+a_1-(a_0+b_0),b)\}$,
then
$$\min_{i \in I_3}~(m_{P_\infty}(i))=d^*+a_1+b_1-2(a_0+b_0)=(a_0+b_0)q-(a_0+b_0).$$\\
If $a_0+b_0 \leq b_1 \leq a_1<q$ and
$I_{3^{'}}=\{(2g-2+a,b),(2g-2+a,b+1),\ldots,(2g-2+a,b+b_1-(a_0+b_0))\}$,
then
$$\min_{i \in I_{3^{'}}}~(m_{P_0}(i))=d^*+a_1+b_1-2(a_0+b_0)=(a_0+b_0)q-(a_0+b_0).$$
\end{lemma}
\begin{proof}
For the case $a_0+b_0 \leq a_1 \leq b_1<q$, we need to show that\\
$m_{P_\infty}(i)\geq d^*+a_1+b_1-2(a_0+b_0)$ for $i \in I_3$, that
is, we need to show that \\
$m_{P_\infty}(2g-2+a_0(q+1)-a_1^{'},b) \geq d^*+a_1+b_1-2(a_0+b_0)$
for $a_1^{'}=a_1, a_1-1, \ldots, a_0+b_0$. By Proposition
\ref{prop:mult}, we have
\begin{eqnarray*}
&&m_{P_\infty}(2g-2+a_0(q+1)-a_1^{'},b)-(d^*+a_1+b_1-2(a_0+b_0))\\
&&=a_1^{'}(q+a_0+b_0-a_1^{'})-a_1^{'}-(a_0+b_0)q+(a_0+b_0)\\
&&=(a_1^{'}-(a_0+b_0))(q-a_1^{'})-(a_1^{'} -(a_0+b_0))\geq 0
\end{eqnarray*}
The other case follows by symmetry.
\end{proof}

In order to prove that the lower bounds of Theorem \ref{thm:dist}
and Theorem \ref{thm:below} are sharp, we need to show that there
exist words that have weight equal to the lower bounds. This can be
shown by constructing functions with certain properties. The
functions consist of multiplications of conics and lines. The
following lemmas show that there are enough conics and lines to
construct such functions.

\begin{lemma}
The curves $y^q+y=x^{q+1}$ and $x^2=\alpha y$ share the following automorphisms,
\[
\sigma(x,y) = (xy^{-1},y^{-1}), \quad \rho_a(x,y) = (ax,a^2y) ~~ \text{for $a \in \mathbb{F}_{q}^\ast.$}
\]
The group generated by the automorphisms is the dihedral group of size $2(q-1).$
\end{lemma}

\begin{proof}
The first claim is easily verified. Finally, $\sigma$ is of order two and
$\sigma \rho_a \sigma = \rho_a^{-1}.$
\end{proof}

\begin{lemma} \label{conD}
For a rational point $P=(u,v)$ with $u \neq 0$ and $v \not \in \mathbb{F}_{q}$, the function
$x^2 - \alpha y$, for $\alpha = u^2 / v $, has zeros in $P_0$ (with multiplicity two)
and in $2q-2$ other rational points including $P$. The number of such functions is
$(q^2-1)/2$ when $q$ is odd and $(q^2+q)/2$ when $q$ is even.
\end{lemma}

\begin{proof}
The function $x^2 - \alpha y$ has poles only at $P_\infty$ of order $2q$. Thus, there are
$2q$ zeros, two of which are $P_0$. We claim that the remaining zeros form a single
orbit under the action of the dihedral group in the previous lemma. The orbit
includes the points
\[
\{ (au,a^2v), (auv^{-1}, a^2v^{-1}) : a \in \mathbb{F}_{q}^\ast \}.
\]
For $P=(u,v)$ with $u \neq 0$, the set $\{ (au,a^2v) : a \in \mathbb{F}_{q}^\ast \}$
consists of $q-1$ distinct points. To show that the second group of $q-1$ points is disjoint
from the first group it suffices to show that $(uv^{-1},v^{-1}) \not \in \{ (au,a^2v) : a \in \mathbb{F}_{q}^\ast \}$.
But $(uv^{-1},v^{-1}) = (au,a^2v)$ if and only if $v=a^{-1}$ which is excluded by the
assumption $v \not \in \mathbb{F}_{q}$. We compute the number of points $N$ such that
$u \neq 0$ and $v \not \in \mathbb{F}_{q}$.
\[
\begin{array}{|c|c|c|c|c|c|}
\hline
 &P=(u,v)   &P : v \in \mathbb{F}_{q} &P : u = 0 &P: v \in \mathbb{F}_{q} \wedge u = 0  &N \\\hline
\text{$q$ odd}  &q^3 &q^2 &q &1 &q^3-q^2-q+1 \\
\text{$q$ even} &q^3 &q   &q &q &q^3-q\\
\hline
\end{array}
\]
To each point corresponds a unique function, and the number of functions is obtained
as $N/(2q-2).$
\end{proof}
We rewrite Lemma \ref{conD} in terms of divisors in Remark \ref{conic}.
\begin{remark} \label{conic}
Let $x^2-\alpha y$ be a conic over the field $\mathbb{F}_{q^2}$ such
that
\begin{equation}\label{cocond}
(x^2-\alpha y)=2P_\infty+2P_0+P_1+ \cdots +P_{2(q-1)}-2H_\infty,
\end{equation}
where $P_i $'s are distinct $\mathbb{F}_{q^2}$-rational points for $i=1,2,\ldots , 2(q-1)$ and $H_\infty = (q+1)P_\infty$.\\
If $q$ is odd then there are $(q^2-1)/2$ number of conics that
satisfies
(\ref{cocond}).\\
If $q$ is even then there are $(q^2+q)/2$ number of conics that
satisfies
(\ref{cocond}).\\
\end{remark}

\begin{remark}\label{linecond}
For the Hermitian curve $y^q+y=x^{q+1}$ over $\mathbb{F}_{q^2}$, the
line passing through any two rational points intersect the curve in
q+1 distinct rational points. Hence we can choose a line with
divisors as below :
\begin{align*}
&(y-\beta x)= P_0+ q ~\text{distinct points}~ - H_\infty.\\
&(x-\gamma)=P_\infty + q  ~\text{distinct points}~ - H_\infty.\\
&(y-\delta)=q+1 ~\text{distinct points}~ - H_\infty.
\end{align*}
for some $\beta,~~\gamma$ and $\delta$ in $\mathbb{F}_{q^2}$.
\end{remark}
\ \\ \noindent {\bf Theorem 3.3} {\it Suppose that $G$ satisfies either
\begin{itemize}
\item[$(1)$]$\rm{deg} G > \rm{deg} K + q$ or
\item[$(2)$]$\rm{deg} K \leq \rm{deg} G \leq \rm{deg} K+q ~\text{and}~ G \nsim sP_\infty ~\text{and} ~G \nsim tP_0$ for all $s,t \in \mathbb{Z}$.
\end{itemize}
Let $G = K+aP_\infty+bP_0$, where $K$ is a
canonical divisor,
\begin{eqnarray*}
a&=& a_0 (q+1)-a_1,~~~~0\leq a_1\leq q,\\
b&=& b_0 (q+1)-b_1,~~~~0\leq b_1\leq q.
\end{eqnarray*}
Let $d^*=\deg(G) -(2g-2)=a+b$.
\begin{itemize}
\item[$1$.] If $0\leq a_1,b_1\leq a_0+b_0$, then\\
$d(C(D,G)^\perp)=d^*$.
\item[$2$.] If $0\leq b_1\leq a_0+b_0<a_1$, then\\
$d(C(D,G)^\perp)=d^*+a_1-(a_0+b_0)$.
\item[$2^{'}$.] If $0\leq a_1\leq a_0+b_0<b_1$, then\\
$d(C(D,G)^\perp)=d^*+b_1-(a_0+b_0)$.
\item[$3$.] If $a_0+b_0<a_1\leq b_1<q$, then\\
$d(C(D,G)^\perp)=d^*+a_1+b_1-2(a_0+b_0)$.
\item[$3^{'}$.] If $a_0+b_0<b_1\leq a_1<q$, then\\
$d(C(D,G)^\perp)=d^*+a_1+b_1-2(a_0+b_0)$.
\item[$4$.] If $a_0+b_0<a_1,b_1$ and $a_1=q$, $b_1=q$, then\\
$d(C(D,G)^\perp)=d^*+q-(a_0+b_0)$.
\end{itemize}}

\begin{figure}[!ht] \label{blockfig}
\centering \psfrag{q}{$q$} \psfrag{b1}{$b_1$} \psfrag{0}{$0$}
\psfrag{a1}{$a_1$} \psfrag{Case 1}{$Case 1$} \psfrag{Case 3}{$Case
3$} \psfrag{Case 2}{$Case 2$} \psfrag{Case 2'}{$Case 2^{'}$}
\psfrag{Case 3'}{$Case 3^{'}$} \psfrag{Case4}{$Case4$}
\includegraphics[scale=0.7]{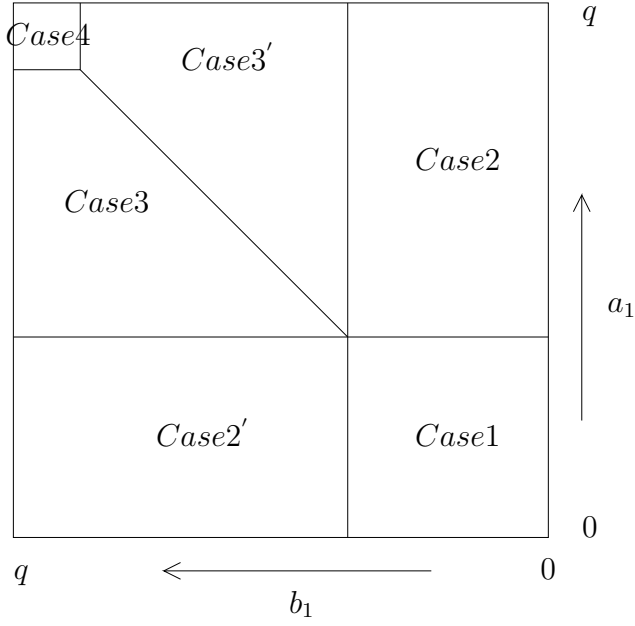}
\caption{Cases for Theorem \ref{thm:dist}}
\end{figure}
\begin{proof}
To prove the theorem we find a path for each case which will give a
lower bound of the minimum distance. Then we show that it is sharp
by
finding a word of weight equal to the lower bound.\\
Case $1$.\\
We fix $b$ and increase $a$. By Theorem \ref {thm:mult},
$m_{P_\infty}(2g-2+a,b)=d^*=a+b$ and $m_{P_\infty}(2g-2+a',b')\geq
a'+ b'$ for arbitrary $a'$ and $b'$. Thus the lower bound is
$d^*$.\\
Case $2$.\\
We fix $b$ and increase $a$ to $a+a_1-(a_0+b_0)$. Then Case $2$ is
reduced to Case $1$. By Lemma \ref{mult2}, the lower bound is
$d^*+a_1-(a_0+b_0)$.\\
Case $2^{'}$.\\
By the symmetry of Case $2$, the lower bound is
$d^*+b_1-(a_0+b_0)$. \\
Case ${3}$.\\
We fix $b$ and increase $a$ to $a+a_1-(a_0+b_0)$. Then Case $3$ is
reduced to Case $2^{'}$. By Lemma \ref{mult3}, the lower bound is
$d^*+a_1+b_1-2(a_0+b_0)$.\\
Case $3^{'}.$\\
By the symmetry of Case $3$, the lower bound is
$d^*+a_1+b_1-2(a_0+b_0)$. \\
Case $4$.\\
We fix $b$ and increase $a$ by 1. Then Case $4$ is reduced to Case
$3$. By Proposition \ref{prop:mult},
\begin{align*}
&m_{P_\infty}(2g-2+a,b)=q(a_0+b_0)-q ~ \text{ and} \\
&m_{P_\infty}(2g-2+a+1,b)=(q-1)(a_0+b_0+1)-(q-1)\\
&~~~~~~~~~~~~~~~~~~~~~~~~~~~~~~= q(a_0+b_0)-(a_0+b_0).
\end{align*}
Thus the lower bound is obtained at $m_{P_\infty}(2g-2+a,b)$. In
order to prove that the lower bound of the code using the points $P
\in X(\mathbb{F}_{q^2})\backslash \{P_0, P_\infty \}$ is sharp, we
need to find a word of weight $d$, where $d$ is the lower bound.
We use the fact
that there exist a word with support $P_1,P_2,\ldots, P_d$ if and
only if $\Omega(G-P_1\cdots -P_d)\neq \Omega(G)$. Equivalently,
$$L(P_1 + \cdots+ P_d-aP_\infty-bP_0) \neq L(-aP_\infty-bP_0).$$
Note that $L(-aP_\infty-bP_0)=0$ because $a+b>1$. We need to find
$P_1,P_2,\ldots, P_d$ such that
$$P_1 +
\cdots+ P_d+a_1P_\infty+b_1P_0 \sim (a_0+b_0)H_\infty
+E,~~\text{where}~~ E \geq 0.$$
 Case  $2^{'}$ reduces to Case $1$ by taking
$E=(b_1-(a_0+b_0))P_0$. Case $2$ reduces to Case $1$ by taking
$E=(a_1-(a_0+b_0))P_\infty$. Case $3$ and $3^{'}$ reduces to Case
$1$ by taking $E=(a_1-(a_0+b_0))P_\infty + (b_1-(a_0+b_0))P_0$. Thus
all cases reduce to case $1$ except case $4$ which we prove separately.\\
Case $4$. $a_0+b_0<a_1,b_1$ and $a_1=q$, $b_1=q$.\\
We need to find $P_1,P_2,\ldots, P_d$, where $d=(a_0+b_0-1)q$, such
that
\begin{equation}\label{upperboundeq}
P_1 + \cdots+ P_d+qP_\infty + qP_0 \sim (a_0+b_0)H_\infty +E,~~\text{
where}~~ E \geq 0.
\end{equation}
Let $(x)=P_\infty+P_0+P_1+P_2+ \cdots + P_{q-1}-H_\infty$. We take
$$f=\frac{x}{y}\prod_{i=1}^{a_0+b_0-1}\frac{1}{y-y_i},$$ where $y_i$ is the
$y$-coordinate of $P_i$.\\ We have
\begin{align*}
&\left(\prod_{i=1}^{a_0+b_0-1}\frac{1}{y-y_i}\right)=(a_0+b_0-1)H_\infty
-P_1-P_2-\cdots-P_{a_0+b_0-1}-\cdots -P_{(a_0+b_0-1)(q+1)} ~~
\text{and}\\
&\left( \frac{x}{y} \right)= P_\infty + P_0 + P_1 + ... + P_{q-1} -
(q+1)P_0.
\end{align*}
Hence
$$(f)=(a_0+b_0)H_\infty - qP_\infty - qP_0- (a_0+b_0-1)q ~~\text{distinct
points}  +P_{a_0+b_0} + \cdots + P_{q-1}.$$ Therefore
(\ref{upperboundeq}) is satisfied with $E = P_{a_0+b_0} +\cdots +
P_{q-1}.$\\
Case $1$. $0\leq a_1,b_1\leq a_0+b_0$.\\
We need to show that there exist a function in
\begin{align*}
&L(P_1 +\cdots+ P_d+a_1P_\infty + b_1P_0 - (a_0+b_0)H_\infty),~~ \text{
or}\\
&L((a_0+b_0)H_\infty - P_1 -\cdots - P_d-a_1P_\infty - b_1P_0 ).
\end{align*}
We construct $f$ in $L((a_0+b_0)H_\infty - P_1 -\cdots - P_d-a_1P_\infty -
b_1P_0 )$ by using the functions $x^2-\alpha_{i}y,~~ y-\beta_jx,~~
x-\gamma_k$, and $~y-\delta_l$. Since $0+K+aP_\infty+bP_0 \sim K+D$,
$a+b=q^3-1$. Then $a_0+b_0 \leq q^2-q-1$. Thus we need at most
$(q^2-q-1)/2$ conics that satisfy the condition (\ref{cocond}) of
Remark \ref{conic}. By Lemma \ref{conD}, there are enough conics
that satisfy (\ref{cocond})
which can be used to construct the function $f$.\\
Case : $a_1 \leq b_1 \leq a_0+b_0$.\\
If $a_1=2m$ is even then, take
$$f=\prod_{i=1}^{m}(x^2-\alpha_{i}y) \times \prod_{i=1}^{b_1-a_1}(y-\beta_ix) \times \prod_{i=1}^{a_0+b_0-b_1}(y-\delta_i).$$
If $a_1=2m+1$ is odd, then take
$$f=x\prod_{i=1}^{m}(x^2-\alpha_{i}y) \times \prod_{i=1}^{b_1-a_1}(y-\beta_ix) \times \prod_{i=1}^{a_0+b_0-b_1}(y-\delta_i).$$
Case: $b_1 < a_1$.\\
If $b_1=2m$ is even then, take
$$f=\prod_{i=1}^{m}(x^2-\alpha_{i}y) \times \prod_{i=1}^{a_1-b_1}(x-\gamma_i) \times \prod_{i=1}^{a_0+b_0-a_1}(y-\delta_i).$$
If $b_1=2m+1$ is odd, then take
$$f=x\prod_{i=1}^{m}(x^2-\alpha_{i}y) \times \prod_{i=1}^{a_1-b_1}(x-\gamma_i) \times \prod_{i=1}^{a_0+b_0-a_1}(y-\delta_i).$$
\end{proof}
\ \\ \noindent {\bf Theorem 3.5} {\it Suppose that $G$ satisfies either
\begin{itemize}
\item[$(1)$]$\rm{deg} G < \rm{deg} K$ or
\item[$(2)$]$\rm{deg} K \leq \rm{deg} G \leq \rm{deg} K+q ~\text{with}~ G \sim sP_\infty ~\text{or}~ G \sim tP_0$ for some $s,t \in \mathbb{Z}$.
\end{itemize}
If $G = aP_\infty+bP_0$ with
\begin{eqnarray*}
a&=& a_0 (q+1)+a_1,~~~~0\leq a_1\leq q,\\
b&=& b_0 (q+1)+b_1,~~~~0\leq b_1\leq q.
\end{eqnarray*}
then $$d(C(D,G)^\perp)=a_0+b_0+2.$$}
\begin{proof}
We may assume that $K=(q-2)(q+1)P_\infty$. Let $H_0 = (q+1)P_0$.\\
It suffices to prove for the three cases (i) $a_0+b_0\leq q-3$, (ii) $a_0+b_0=q-2,~~a_1=0$, and
(iii) $ a_0+b_0=q-2,~~b_1=0$. If we assume $b_0=0$ then (i) contains the case $\rm{deg} G < \rm{deg} K$, (ii) contains the case $\rm{deg} K \leq \rm{deg} G \leq \rm{deg} K+q ~\text{with}~ G \sim tP_0$ and (iii) contains the case $\rm{deg} K \leq \rm{deg} G \leq \rm{deg} K+q ~\text{with}~ G \sim sP_\infty$.\\
Case 1.  $a_0+b_0\leq q-3$\\
 Let $G_1
=a_0H_\infty+b_0H_0$ and $G_2 =a_0H_\infty+qP_\infty+b_0H_0+qP_0$.
By Proposition \ref{prop:mult}, we have that $d(C(D,G_1)^\perp)\geq
a_0+b_0+2$. Since $C(D,G_2)^\perp \subseteq C(D,G_1)^\perp$,
$~d(C(D,G_2)^\perp)\geq a_0+b_0+2$. We need to show that there
exists a word of weight $a_0+b_0+2$ in  $C(D,G_2)^\perp$. The code
$C(D,G_2)^\perp=C_\Omega(D,G_2)$ has a word of weight $d=a_0+b_0+2$
if there exists $P_1,\ldots,P_d$ with $\Omega(G_2-P_1\cdots
-P_d)\neq \Omega(G_2)$. Equivalently, if
\begin{eqnarray*}
&&L((q-2)H_\infty-(a_0+b_0+2)H_\infty+P_0+P_1+\cdots+P_d)\\
&&\neq L((q-2)H_\infty-(a_0+b_0+2)H_\infty+P_\infty+P_0).
\end{eqnarray*}
We choose $P_1,\ldots,P_d$ on a line that pass through $P_0$ and
$P_\infty$. Since $P_1+ \cdots +P_d+P_{d+1}+ \cdots
+P_{q-1}+P_0+P_\infty \sim (q+1)P_\infty$, it is enough to show that
\begin{eqnarray*}
&&L((q-1)H_\infty-(a_0+b_0+2)H_\infty-P_{d+1}
\cdots -P_{q-1})\\
&&\neq L((q-1)H_\infty-(a_0+b_0+2)H_\infty-P_1-P_2 \cdots -P_{q-1}).
\end{eqnarray*}
We take $f=\prod_{i=d+1}^{q-1}(y-y_i)$, where $y_i$ is the
$y$-coordinate of $P_i$.\\
 In both cases $a_0+b_0=q-2,~~a_1=0$ and
$a_0+b_0=q-2,~~b_1=0$, we have that the lower bound is $d=a_0+b_0+2$
by Proposition \ref{prop:mult}. Now we
show that the lower bound is sharp for both cases.\\
Case 2.
$a_0+b_0=q-2,~~a_1=0$\\
Since $G=K+b_1P_0$, we need to show that there exist
$P_1,\ldots,P_d$ with
\begin{eqnarray*}
&&L(P_1 + \cdots+ P_d-b_1P_0)\\
&&\neq L(-b_1P_0).
\end{eqnarray*}
Take $f=y/(y-x)$.\\
Case 3. $a_0+b_0=q-2,~~b_1=0$\\
Since $G=K+a_1P_\infty$, we need to show that there exist
$P_1,\ldots,P_d$ with
\begin{eqnarray*}
&&L(P_1 + \cdots+ P_d-a_1P_\infty)\\
&&\neq L(-a_1P_\infty).
\end{eqnarray*}
Take $f=x-1$.
\end{proof}

\section{Appendix}

In this section we give formulas for the minimum distance obtained
by Homma and Kim. Also, we give formulas for the
minimum distance obtained by our method for comparison.\\

\noindent {\bf Homma and Kim method}\\
Let $X$ be a Hermitian curve defined by $y^q+y=x^{q+1}$ over
$\mathbb{F}_{q^2}$. Let $P_\infty$ be the point at infinity of $X$
and $P_0$ the origin of $X$. We consider the code $C(m,n)$ in
$(\mathbb{F}_{q^2})^{q^3-1}$ defined by the image of the evaluation
map
\begin{eqnarray*}
L(mP_{\infty}+nP_0)&\longrightarrow&(\mathbb{F}_{q^2})^{q^3-1}\\
        f ~~~~~~~~ &\longmapsto&(f(P))_{P \in
        X(\mathbb{F}_{q^2})\setminus \{P_{\infty},~P_0\}}~,
\end{eqnarray*}
where $X(\mathbb{F}_{q^2})$ denotes the set of
$\mathbb{F}_{q^2}$-rational points of $X$. Our problem is to
determine the minimum distance of $C(m,n)$ for $0 \leq n \leq q$.\\
For $n=0$ the following theorem holds.
\begin{theorem}\cite[Theorem 5.2]{HommaKim4}
Let $m=aq+b=(q^2-\rho)q+b\in A(\mathbb{Z},q)$ with
 $b\leq a\leq
\min\{b+q^2-1, q^2+q-3\}$.
\begin{itemize}
\item[(i)] If $b\leq a\leq b+q^2-q-1$, then $d(C(m,0))=q^3-1-m$.
\item[(ii)] If $1\leq \rho\leq q$ and $0\leq b\leq q-\rho$, then
$d(C(m,0))=\rho q-1$.
\item[(iii)] If $q^2-1\leq a\leq \min\{b+q^2-1,q^2+q-3\}$, then
$d(C(m,0))=q^2+q-a-2$.
\end{itemize}
\end{theorem}

\noindent For $n$ with $1\leq n\leq q-1$, we have the following
theorems.

\begin{theorem}\cite{HommaKim4}, \cite{HommaKim3}, \cite{HommaKim2}
Fix an integer $n$ with $1\leq n\leq q-1$. Let $m=aq+b$ be a
nonnegative integer with $0\leq b< q$.
\begin{itemize}
\item[[I\hspace{-0.2cm}]] If $b\leq a\leq q-(n+1)$, then $d(C(m,n))=q^3-1-m$.
\item[[II\hspace{-0.2cm}]] If $m$ satisfies either
\begin{itemize}
\item[(i)] $0\leq b\leq q-2$ and $ q^2-1\leq a \leq b+q^2-1$, or
\item[(ii)] $b=q-1$ and $q^2-1\leq a\leq q^2+q-(n+3)$,
\end{itemize}
then $d(C(m,n))=q^2+q-a-2$.
\item[[III\hspace{-0.2cm}]] If $m$ satisfies either
\begin{itemize}
\item[(i)] $b=0$ and $q-n\leq a\leq q^2-(n+1)$, or
\item[(ii)] $1\leq b\leq q-2$ and $\max\{b,q-n\}\leq a \leq\min\{b+q^2-
(q+1),q^2-(n+2)\}$, or
\item[(iii)] $b=q-1$ and $q-(n+1)\leq a\leq q^2-(n+2)$, then
$d(C(m,n))=q^3-1-(m+n)$.
\end{itemize}
\end{itemize}
\end{theorem}

\noindent In order to describe $d(C(m,n))$ for remaining $m$, we put
$m=(q^2-\rho)q+b$ for convenience.

\begin{theorem}\cite[Theorem 1.4]{HommaKim1}
Fix an integer $n$ with $1\leq n\leq q-1$. Let $m=(q^2-\rho)q+b$ be
an integer with $0\leq b<q$.
\begin{itemize}
\item[[IV\hspace{-0.2cm}]] If $1\leq b$, $n+1\leq \rho$ and
$\rho+b\leq q$, then $d(C(m,n))=\rho q-(n+1)$.
\item[[V\hspace{-0.2cm}]] If $\rho\leq n+1$ and $q<\rho+b$, then
$d(C(m,n))=\rho(q-1)-(b-1)$.
\item[[VI\hspace{-0.2cm}]] Assume that $2\leq \rho\leq n$ and
$\rho+b\leq q$.
\begin{itemize}
\item[[VI-1\hspace{-0.2cm}]] If either ``$n\leq q-2$'' or
``$n=q-1$ and $\rho+b< q$'', then $d(C(m,n))=\rho(q-1)$.
\item[[VI-2\hspace{-0.2cm}]] If $n=q-1$ and $\rho+b=q$, then
$d(C(m,n))=(\rho-1)q$.
\end{itemize}
\end{itemize}
\end{theorem}


We denote by $A(\mathbb{Z},q)$ the array of integers with the
infinite length of column
\begin{equation*}
\begin{array}{cccc}
\vdots & \vdots & {} &\vdots\\
-q & -q+1&\ldots & -q+(q-1)\\
0 & 1& \ldots & (q-1)\\
q & q+1 & \ldots & q+(q-1)\\
2q & 2q+1 & \ldots & 2q+(q-1)\\
\vdots & \vdots & {} &\vdots
\end{array}
\end{equation*}
Let
\begin{equation*}
\tilde{I}_q =\{aq +b \in A(\mathbb{Z},q)|b\leq a\}\cup \{aq
+(q-1)|0\leq a\leq q-2\}\cup \{-1\}
\end{equation*}
and
\begin{equation*}
J_q=\{aq+b\in A(\mathbb{Z},q)|b+q^2\leq a\}\cup\{aq+(q-1)|q^2-2\leq
a\}.
\end{equation*}

\noindent The following is the formula for the minimum distance of
$C(m,q)$.

\begin{theorem}\cite[Theorem 6.1]{HommaKim4}
Let $m=aq+b ~(0\leq b < q)$ be an integer in $\tilde{I}_q\backslash
J_q$.
\begin{itemize}
\item[(A)] If $m$ satisfies either
\begin{itemize}
\item[(i)] $0\leq b\leq q-2$ and $b\leq a\leq b+q^2-q-1$, or
\item[(ii)] $b=q-1$ and $-1\leq a\leq q^2-3$,
\end{itemize}
then $d(C(m,q))=q^3-q-m-1$.
\item[(B)] If $m$ satisfies the condition
$$
b+q^2-q\leq a\leq q^2-2,
$$
then $d(C(m,q))=(q^2-a-1)q$.
\item[(C)] If $m$ satisfies the condition
$$
0\leq b\leq q-2~~~\mbox{and}~~~q^2-1\leq a\leq b+q^2-1,
$$
then $d(C(m,q))=q^2+q-a-2$.
\end{itemize}
\end{theorem}

\noindent{\bf Formulas using our method}
\begin{theorem}\label{dist}
Suppose that $G$ satisfies either
\begin{itemize}
\item[$(1)$]$\rm{deg} G > \rm{deg} K + q$ or
\item[$(2)$]$\rm{deg} K \leq \rm{deg} G \leq \rm{deg} K+q ~\text{and}~ G \nsim sP_\infty ~\text{and} ~G \nsim tP_0$ for all $s,t \in \mathbb{Z}$.
\end{itemize}
Let $G = K+aP_\infty+bP_0$, where $K$ is a
canonical divisor,
\begin{eqnarray*}
a&=& a_0 (q+1)-a_1,~~~~0\leq a_1\leq q,\\
b&=& b_0 (q+1)-b_1,~~~~0\leq b_1\leq q.
\end{eqnarray*}
Let $d^*=\deg(G) -(2g-2)=a+b$.
Then
$$
d(C(D,G)^\perp)=d^*+\max\{0,a_1-(a_0+b_0),
b_1-(a_0+b_0),a_1+b_1-2(a_0+b_0)\},
$$
except for the case when $(a_0+b_0<a_1,b_1$ and $a_1=q$, $b_1=q)$,
for which
$$d(C(D,G)^\perp)=d^*+q-(a_0+b_0).$$
\end{theorem}
\begin{theorem}\label{below}
Suppose that $G$ satisfies either
\begin{itemize}
\item[$(1)$]$\rm{deg} G < \rm{deg} K$ or
\item[$(2)$]$\rm{deg} K \leq \rm{deg} G \leq \rm{deg} K+q ~\text{with}~ G \sim sP_\infty ~\text{or}~ G \sim tP_0$ for some $s,t \in \mathbb{Z}$.
\end{itemize}
If $G = aP_\infty+bP_0$ with
\begin{eqnarray*}
a&=& a_0 (q+1)+a_1,~~~~0\leq a_1\leq q,\\
b&=& b_0 (q+1)+b_1,~~~~0\leq b_1\leq q.
\end{eqnarray*}
then $$d(C(D,G)^\perp)=a_0+b_0+2.$$
\end{theorem}
\begin{example}Let $X$ be a Hermitian curve defined by
$y^8 + y = x^9$ over $\mathbb{F}_{64}$. Let $K$ be a canonical
divisor and $G=mP_{\infty}+nP_0$. We consider the code
$C(m,n)^\perp$. We give two tables with $\rm{deg}G < \rm{deg}K$ and $\rm{deg}G >
\rm{deg}K+q$. The rows represent $m$ and the columns represent
$n$. The entries of the first matrix are the minimum distance of
$C(m,n)^\perp$ and the entries of the second matrix state which
formula from Homma and Kim were used to find the minimum distance.
If $G=82P_{\infty}+3P_0$ then the minimum distance of
$C(82,3)^{\perp}=35$ and since the $(82,3)$ entry of the second
matrix is 361 it means Theorem 6.3 VI-1 was used to find the minimum
distance. The entries with zero mean that it is not in the range of
Homma and
Kim's formula.\\

Table \ref{tab:first} is an example of cases when $\rm{deg}G < \rm{deg}K$.

\begin{table}[h]
\begin{center}
\renewcommand{\arraystretch}{1.5}
\begin{tabular}{|ccccccccc|c|p{0.4cm}p{0.4cm}p{0.4cm}p{0.4cm}p{0.4cm}p{0.4cm}p{0.4cm}p{0.4cm}p{0.4cm}|}
\cline{1-9}\cline{11-19}
~4&~4&~4&~4&~4&~4&~4&~4&~4~~&~&~0&~~0&~~0&~~0&~~0&~~0&~~0&~~0&~0\\
~4&~4&~4&~4&~4&~4&~4&~4&~4~~&~&~0&~~0&~~0&~~0&~~0&~~0&~~0&~~0&~0\\
~4&~4&~4&~4&~4&~4&~4&~4&~4~~&~&~0&~~0&~~0&~~0&~~0&~~0&~~0&~~0&~0\\
~4&~4&~4&~4&~4&~4&~4&~4&~4~~&~&~0&~~0&~~0&~~0&~~0&~~0&~~0&~~0&~0\\
~4&~4&~4&~4&~4&~4&~4&~4&~4~~&~&~0&~~0&~~0&~~0&~~0&222&222&222&13\\
~4&~4&~4&~4&~4&~4&~4&~4&~4~~&~&~0&221&221&221&221&221&221&221&13\\
~4&~4&~4&~4&~4&~4&~4&~4&~4~~&~&~0&221&221&221&221&221&221&221&13\\
~4&~4&~4&~4&~4&~4&~4&~4&~4~~&~&~0&221&221&221&221&221&221&221&13\\
~4&~4&~4&~4&~4&~4&~4&~4&~4~~&~&~0&221&221&221&221&221&221&221&13\\
\cline{1-9}\cline{11-19}
\end{tabular}
\end{center}
\caption{Minimum
distance for codes $C(m,n)^\perp$ when $m=18, \ldots , 26$ and $n=0,\ldots, 8.$ } \label{tab:first}
\end{table}

Table \ref{tab:second} is an example of cases when $\rm{deg}G >
\rm{deg}K+q$. The matrices are given in 10 by 10 for easy comparison.
 Homma and Kim used $n=0,\ldots,q$ and
$m=aq + b$ where $b=0,\ldots,q$ which is the upper left 9 by 9
matrix. We used $m=a_0(q+1)-a_1$ and $n=b_0(q+1)-b_1$, where $ 0
\leq a_1, b_1 \leq q $ which is the
lower right 9 by 9 matrix.

\begin{table}[h]
\begin{center}
\renewcommand{\arraystretch}{1.5}
\begin{tabular}{|p{0.33cm}|p{0.33cm}p{0.33cm}
p{0.33cm}p{0.33cm}p{0.33cm}p{0.33cm}p{0.33cm}p{0.33cm}|p{0.33cm}|c|p{0.49cm}|p{0.35cm}p{0.35cm}
p{0.35cm}p{0.35cm}p{0.35cm}p{0.35cm}p{0.35cm}p{0.43cm}|p{0.49cm}|}
\cline{1-10}\cline{12-21}
\multicolumn{2}{|p{0.33cm}}{27~~~32}&32&32&32&32&33&34&35&36&&
\multicolumn{2}{|p{0.49cm}}{411~~~~35}&~35&~35&~35&232&232&232&~11&411\\
\cline{2-10}\cline{13-21}
32&32&35&35&35&36&37&38&39&40&&~42&362&361&361&~34&~34&~34&~34&~12&~42\\
32&35&35&35&35&36&37&38&39&40&&~42&361&361&361&~34&~34&~34&~34&~12&~42\\
32&35&35&35&35&36&37&38&39&40&&~42&361&361&361&~34&~34&~34&~34&~12&~42\\
32&35&35&35&35&36&37&38&39&40&&~42&361&361&361&231&231&231&231&~12&412\\
32&36&36&36&36&37&38&39&40&41&&412&~35&~35&~35&233&233&233&233&~11&411\\
33&37&37&37&37&38&39&40&41&42&&411&~35&~35&~35&232&232&232&232&~11&411\\
34&38&38&38&38&39&40&41&42&43&&411&~35&~35&~35&232&232&232&232&~11&411\\
35&39&39&39&39&40&41&42&43&44&&411&~35&~35&~35&232&232&232&232&~11&411\\
\cline{1-9}\cline{12-20}
36&40&40&40&40&41&42&43&\multicolumn{2}{p{0.33cm}|}{44~~~45}&&
411&~35&~35&~35&232&232&232&232&\multicolumn{2}{p{0.43cm}|}{~11~~~~411}\\
\cline{1-10}\cline{12-21}
\end{tabular}\\
\end{center}
\caption{Minimum distance for codes $C(m,n)^\perp$ when
$m=81, \ldots , 90$ and $n=0,\ldots, 9.$ } \label{tab:second}
\end{table}

\end{example}

\newpage
\bibliographystyle{plain}
\bibliography{Hermitianbib}

\end{document}